\documentclass{article}
\usepackage{amsmath,amssymb,amsthm}
\usepackage{fullpage}
\usepackage[dvipdfmx]{graphicx,xcolor}

\DeclareMathOperator{\Ker}{Ker}
\newtheorem{rem}{Remark}[section]{\bfseries}{\itshape}
\newtheorem{defi}{Definition}[section]{\bfseries}{\itshape}
\newtheorem{theo}{Theorem}[section]{\bfseries}{\itshape}
{\bfseries}{\itshape}
\newtheorem{prop}{Proposition}[section]{\bfseries}{\itshape}
{\bfseries}{\itshape}
\newtheorem{cor}{Corollary}[section]{\bfseries}{\itshape}
{\bfseries}{\itshape}
{\bfseries}{\itshape}
{\bfseries}{\itshape}
{\bfseries}{\itshape}
{\bfseries}{\itshape}
{\bfseries}{\itshape}
{\bfseries}{\itshape}
{\bfseries}{\itshape}

\begin{document}

\title{Bounds on Heights of $2$-isogeny Graphs \\
in Ordinary Curves over $\mathbb{F}_p$ and $\mathbb{F}_{p^2}$ 
and Its Application}
\author{Yuji Hashimoto\\ Tokyo Denki University / AIST\\ \texttt{y.hashimoto@mail.dendai.ac.jp} \and Koji Nuida\\ Kyushu University / AIST\\ \texttt{nuida@imi.kyushu-u.ac.jp}}
\date{\today}

\maketitle

\begin{abstract}
    It is known that any 
    isogeny graph consisting of ordinary elliptic curves over $\mathbb{F}_q$ with $q = p$ or $p^2$ has a special structure, called a volcano graph. We have a bound $h < \log_2 \sqrt{4q}$ of a height $h$ of the $2$-volcano graph. 
    In this paper, we improve the bound on a height of $2$-volcano graphs over $\mathbb{F}_q$.
    In case $q = p^2$, we show a tighter bound 
    $h \leq \left\lfloor \frac{ 1 }{ 2 } \lfloor \log_2 p \rfloor \right\rfloor + 2 $. In case $q = p$, we also show that a good bound for each prime $p$ can be computed by using our proposed techniques. 
\end{abstract}

\section{Introduction}
Let $\mathbb{F}_q$ be a finite field with characteristic $p$ and $A$ be an Abelian variety.
An isogeny graph $G\left(A / \mathbb{F}_q \right)$ consists of 
vertices and edges, where each vertex is an isomorphism class of Abelian variety and each edge is an isogeny between two Abelian varieties. 
Structures of isogeny graphs $G\left(A / \mathbb{F}_q \right)$ are
studied for mathematical interest \cite{Sut13, KT20, FS21a, FS21b, BCP22}.  
In particular, isogeny graphs $G\left(E / \mathbb{F}_q \right)$ of elliptic curves $E$ (i.e., Abelian varieties of dimension $1$) have interesting structures.
Isogeny graphs $G\left(E / \mathbb{F}_q \right)$ are 
also applied to isogeny-based cryptography \cite{Cou06, RS06, CLG09, CMPR18, CK20, FMP23}.
Elliptic curves over $\mathbb{F}_q$ are classified into ordinary curves and supersingular curves, and structures of isogeny graphs are 
different depending on whether $E$ is an ordinary curve or supersingular one. 
In detail, ordinary elliptic curves over $\mathbb{F}_{q}$ with $q = p$ or $p^2$ have a graph structure called volcano graph.
On the other hand, supersingular elliptic curves over $\mathbb{F}_{p^2}$ have a graph structure called Ramanujan graph. 
From the perspective of constructing secure isogeny-based cryptosystems, structures of their isogeny graphs $G\left(E / \mathbb{F}_q \right)$ are important.
Thus, structures of isogeny graphs have been studied 
from both mathematical and cryptographic points of view.

In this paper, we focus on ordinary elliptic curves over $\mathbb{F}_q$ with $q = p$ or $q = p^2$.
In particular, we discuss $2$-volcano graph, where each vertex is an isomorphism class of ordinary elliptic curves and each edge is an isogeny with degree $2$.
There is a parameter of $2$-volcano graphs called height. 
The heights of $2$-volcano graphs for ordinary elliptic curves are important from the viewpoint of supersingularity testing, a problem of determining whether a given elliptic curve is ordinary or supersingular.
In detail, in Sutherland's supersingularity testing algorithm \cite{Sut12c} and its improved versions \cite{HT21,HN21}, we search for a terminal vertex in the isogeny graph by drawing a path in the graph, and the curve is ordinary if a terminal vertex is found, while it is supersingular if a terminal vertex is not found.
Now the heights of $2$-volcano graphs give an upper bound for a maximum number of vertices to be searched for finding a terminal vertex in the ordinary case.
Accordingly, a tighter upper bound for the heights of $2$-volcano graphs results in reducing the number of steps in the supersingularity testing algorithm.

\subsection{Our Result}
It is known that the height $h$ of a $2$-volcano graph containing an ordinary elliptic curve defined over $\mathbb{F}_q$ is bounded as $h < \log_2 \sqrt{4q}$  \cite{Sut12c}.
In this paper, we improve this bound.

In case of $q = p^2$, the known bound is $h < \log_2 (2p) = \log_2 p + 1$, therefore $h \leq \lfloor \log_2 p \rfloor + 1$.
We improve this bound to $h \leq \lfloor \frac{1}{2} \lfloor \log_2 p \rfloor \rfloor + 2$.
That is, we reduce the existing bound by about half.

In case of $q = p$, we have a constant bound as follows.
\begin{itemize}
    \item 
    When $p \equiv 3 \pmod{4}$, we have $h \leq 1$.
    \item
    When $p \equiv 5 \pmod{8}$, we have $h \leq 2$.
\end{itemize}
We propose a new technique to computing a bound of a height $h$ for each $p \equiv 1 \pmod{8}$ (hence $p \geq 17$). 
We also experimentally investigate our new bound by using Magma Computational Algebra System.
For example, an average bound for heights for $100$ primes of $1024$-bit length is $258.05$.

Our results can be applied to supersingularity testing algorithms and may be applied algorithm to solving inverse volcano problem.
For the supersingularity testing algorithms in \cite{Sut12c, HT21, HN21}, the maximum number of steps depends on an upper bound for the heights $h$ of $2$-volcano graphs.
By using our result on reducing an upper bound of $h$ by about half, the computational time of the algorithm is also expected to be reduced by about half.
We confirm it by computer experiments.

In addition, our result might be also applicable 
to solving inverse volcano problem \cite{BCP22}.
An inverse volcano problem over $\mathbb{F}_p$ is a problem of finding 
a prime $p$ when the degree $\ell$, the height $h$, and the shape of the surface for the $\ell$-volcano graph over $\mathbb{F}_p$ are specified.
Our new upper bound for the height $h$ could help to search for such a prime $p$.

\section{Elliptic Curves and Isogenies}
In this section, we explain basic points about elliptic curves and isogenies.
For the detail, refer to \cite{S86, Gal12}.
Let $p$ be a prime. 
Let $\mathbb{F}_q$ be a finite field with characteristic $p$ and $\bar{\mathbb{F}}_p$ be an algebraic closure of $\mathbb{F}_p$.

First, we explain definition of elliptic curves.

\begin{defi}\label{EC_def}
For any subfield $\mathbb{F}$ of $\bar{\mathbb{F}}_p$, an elliptic curve defined over $\mathbb{F}$ is a non-singular algebraic curve $E$ with genus one defined over $\mathbb{F}$.
\end{defi}

Specifically, elliptic curves can generally be represented by Weierstrass normal form as follows.
\begin{defi}
For $a_1$, $a_2, \cdots, a_6 \in \mathbb{F}_q$, the following expression of an elliptic curve $E$ is called Weierstrass normal form. 
$$
E : y^2 + a_1 xy + a_3 y = x^3 + a_2 x^2 + a_4 x + a_6 
$$ 
Here, for 
$b_2=a_1^2+4 a_2, b_4=a_1 a_3+2 a_4, b_6=a_3^2+4 a_6, b_8=a_1^2 a_6+4 a_2 a_6-a_1 a_3 a_4+a_2 a_3^2-a_4^2$,
the discriminant $\Delta = -b_2^2 b_8+9 b_2 b_4 b_6-8 b_4^3-27 b_6^2$ of this curve must be non-zero.
\end{defi}

On the other hand, if 
$p \neq 2$, any elliptic curve can be transformed to Legendre form.

\begin{prop}[\text{\cite[Section~3.1]{S86}}]
Every elliptic curve over $\bar{\mathbb{F}}_p$ with $p \neq 2$ is isomorphic to an elliptic curve of Legendre form.
$$
E_{\lambda}: \quad
y^{2}=x(x-1)(x-\lambda)\,\,(\lambda \in \bar{\mathbb{F}}_p,
\lambda\neq 0, 1)
$$
The $j$-invariant of $E_{\lambda}$ is defined by
$$
j(\lambda)=\frac{256\left(\lambda^{2}-\lambda+1\right)^{3}}{\lambda^{2}(\lambda-1)^{2}}.
$$
\end{prop}

Next, we explain definition of isogenies.
\begin{defi}
Two elliptic curves $E, E^{\prime}$ on $\mathbb{F}_q$ are isogenous if there exists a non-constant map $\phi : E\rightarrow E^{\prime}$ such that the following conditions hold.
\begin{enumerate}
\item The $\phi$ is a rational function. That is, for $P=(x, y)\in E$, each coordinate of $\phi(P)\in E^{\prime}$ can be represented by a rational expression of $x, y$.
\item The $\phi$ is homomorphic with respect to addition. That is, for $P, Q\in E$, it satisfies $$\phi(P+Q)=\phi(P)+\phi(Q).$$
Then, $\phi$ is called an isogeny of the elliptic curve $E$.
\end{enumerate}
\end{defi}

An isogeny of degree $\ell$ is defined as follows.

\begin{defi}
For any $\ell$ with $p\mathrel{\not|}\ell$, a separable isogeny $\phi : E\rightarrow E^{\prime}$
is called an $\ell$-isogeny if the kernel $\Ker \phi$ is isomorphic to the cyclic group $\mathbb{Z}/\ell \mathbb{Z}$. Then, $E$ and $E^{\prime}$ are called $\ell$-isogenous.
\end{defi}

\section{Isogeny Volcano Graphs of Ordinary Curves}

In this section, we explain isogeny graphs.
For the details, refer to \cite{Sut13}.

An $\ell$-isogeny graph $G_{\ell}(\mathbb{F}_{q})$ is a graph in which the vertices consist of $\bar{\mathbb{F}}_p$-isomorphism classes (or equivalently, the $j$-invariants) of elliptic curves over $\mathbb{F}_{q}$ and the edges correspond to isogenies of degree $\ell$ defined over $\mathbb{F}_{q}$.

We denote by $G_\ell(E/\mathbb{F}_{q})$ the connected component of $G_{\ell}(\mathbb{F}_{q})$ containing the $j$-invariant $j(E)$ of an elliptic curve $E$ defined over $\mathbb{F}_{q}$. 
We note that the vertex set of a connected component of $G_{\ell}(\mathbb{F}_{q})$ consists of either ordinary curves only or supersingular curves only.
It is known that the connected component $G_{\ell}(E/\mathbb{F}_{q})$ of an isogeny graph at an ordinary elliptic curve $E$ forms an $\ell$-volcano graph of height $h$ for some $h$, defined as follows.

\begin{defi}
[Def.\,1 in \cite{Sut13}, Def.\,1 in \cite{Sut11}]
\label{Dfn:Volcano}
A connected, undirected, and simple graph $V$ is an $\ell$-volcano graph of height $h$ if there exist $h+1$ disjoint subgraphs $V_0, \ldots ,V_h$ (called level graphs) such that any vertex of $V$ belongs to some of $V_0, \ldots, V_h$ and the following conditions hold.
\begin{enumerate}
\item The degree of vertices except for $V_h$ is $\ell +1$ and the degree of vertices in $V_h$ is $1$ when $h>0$ and at most $2$ when $h=0 \,\, ($the degree in this case depends on the form of $V_0)$.
\item The $V_0$ is one of the following; a cycle (of at least three vertices), a single edge (with two vertices), or a single vertex. Moreover, if $h > 0$, then all the other outgoing edges from a vertex in $V_0$ are joined to vertices in $V_1$. $V_0$ is specially called surface.
\item In the case of $h>i >0$, each vertex in the level $i$ graph $V_i$ is adjacent to only one vertex in the level $i - 1$ graph $V_{i-1}$ and all the other outgoing edges are joined to vertices in $V_{i+1}$.
\item If $h > 0$, then each vertex of $V_h$ has only one outgoing edge and it is joined to a vertex in $V_{h-1}$.
\end{enumerate}
\end{defi}

The graph $G_\ell(\mathbb{F}_{q})$ has a connected component of all the $j$-invariants of supersingular curves over $\bar{\mathbb{F}}_p$ \cite{Koh96}. 
Therefore, other connected components consist of $j$-invariants of ordinary curves.
For connected components of ordinary curves, we have 
the following result when a given connected component does not contain a $j$-invariant $0$ nor $1728$.

\begin{prop}
[\cite{Koh96}, Thm.\,7 in \cite{Sut13}]
\label{Thm:Volcano}
Let $V$ be a connected component of $\ell$-volcano graph $G_\ell(\mathbb{F}_q)$ consisting of $j$-invariants of ordinary curves as vertices different from $0,1728$. 
Let ${\cal O}_0$ be an endomorphism ring of an elliptic curve $E$ in the surface of $V$. 
Then the height of $V$ is given by $h = \nu_\ell \left( \left( t^2 - 4q \right)/D_0 \right)/2$, where $D_0 = {\rm disc}({\cal O}_0)$ (see \cite{Sut13} for details about the discriminant $D_0 = {\rm disc}({\cal O}_0)$), $t^2 = {\rm tr}(\pi_E)^2$ where ${\rm tr}(\pi_E)$ is the trace of the $q$-power Frobenius map for $E$, and $\nu_\ell$ is the $\ell$-adic additive valuation.
\end{prop}
We have the following corollary of the above Proposition \ref{Thm:Volcano} and Remark 8 in \cite{Sut13} which discusses on cases of $j(E)=0,1728$.

\begin{cor}[\cite{Sut13}]
\label{Cor:DepthVolcano}
For any connected component of $G_\ell(\mathbb{F}_{q})$ consisting of ordinary curves, its height $h$ satisfies that $h \leq \log_\ell(\sqrt{4q})$.
\end{cor}

\section{Our Results: Heights of $2$-isogeny Graphs in Ordinary Carves}

For any ordinary elliptic curve $E$ defined over a finite field $\mathbb{F}_q$, as shown in Proposition \ref{Thm:Volcano}, the height $h = h(\ell;E/\mathbb{F}_q)$ of the $\ell$-volcano graph containing $E$ satisfies
\[
    h = \frac{ 1 }{ 2 } \nu_{\ell}( (t^2 - 4q) / D_0 )
\]
where $\nu_{\ell}$ denotes the $\ell$-adic additive valuation, $t \in \mathbb{Z}$ is the trace of $E$, and $D_0 \in \mathbb{Z}_{>0}$ is the discriminant of the field $\mathbb{Q}(\sqrt{t^2 - 4q})$.
By the fact above, we have a bound
\[
    h \leq \frac{ e(q;t) }{ 2 } \,,\, \mbox{ where } e(q;t) := \nu_{\ell}( t^2 - 4q ) \enspace.
\]
Moreover, the Hasse bound and the characterization of supersingular elliptic curves in terms of the trace $t$ imply that $|t| \leq 2 \sqrt{q}$ and $t \not\equiv 0 \pmod{p}$.

By the observation above, an upper bound for the height $h$ will be derived once we obtain an upper bound for the value $e(q;t)$. 
In the following, we give a bound for $e(q;t)$ when $\ell = 2$, $q \in \{p,p^2\}$ with odd prime $p$, and $t$ runs over all integers satisfying that $|t| \leq 2 \sqrt{q}$ and $t \not\equiv 0 \pmod{p}$.

\subsection{Our Bound of Height in $q = p^2$}

We consider the case $q = p^2$.
We give a bound for $e(p^2;t) = \nu_2( t^2 - 4p^2 )$ when $t$ runs over all integers satisfying that $|t| \leq 2 \sqrt{q} = 2p$ and $t \not\equiv 0 \pmod{p}$; in particular, $1 \leq |t| < 2p$.

\begin{theo}
\label{thm:bound:Fp-square}
In the current case (with $\ell = 2$ and $q = p^2$), we have
\[
    e(p^2;t) \leq \lfloor \log_2 p \rfloor + 4 \enspace.
\]
Hence the height $h = h(2;E/\mathbb{F}_{p^2})$ of the $2$-volcano graph with vertices defined over $\mathbb{F}_{p^2}$ is bounded by
\[
    h \leq h_2 \,,\, \mbox{ where }
    h_2 := \left\lfloor \frac{ 1 }{ 2 } \lfloor \log_2 p \rfloor \right\rfloor + 2 \enspace.
\]
\end{theo}
\begin{proof}
The bound for $e := e(p^2;t)$ is satisfied when $e \leq 1$; so we suppose that $e \geq 2$.
By the definition of $e$ and the condition $1 \leq |t| < 2p$, we can write $t^2 - 4p^2 = -2^e a$ with some odd integer $a > 0$.
Then we have $t^2 = 4p^2 - 2^e a = 4 (p^2 - 2^{e-2} a) \in 4 \mathbb{Z}$, therefore $t \in 2\mathbb{Z}$.
Write $t = 2t_0$ with $t_0 \in \mathbb{Z}$.
Then $1 \leq |t_0| < p$, $t_0{}^2 = p^2 - 2^{e-2} a$, and
\[
    2^{e-2} a = p^2 - t_0{}^2
    = (p - t_0)(p + t_0) \enspace,
\]
therefore
\[
    \nu_2( p - t_0 ) + \nu_2( p + t_0 )
    = \nu_2( (p - t_0)(p + t_0) )
    = \nu_2( 2^{e-2} a )
    = e - 2 \enspace.
\]
Now, since $(p - t_0) + (p + t_0) = 2p$ and $\nu_2(2p) = 1$ (recall that $p$ is odd, since $\ell = 2$ and $p\mathrel{\not|}\ell$), we have either $\nu_2(p - t_0) \leq 1$ or $\nu_2(p + t_0) \leq 1$.
Therefore $\nu_2(p + t_0) \geq e - 3$ or $\nu_2(p - t_0) \geq e - 3$.
Hence we can write $p \pm t_0 = 2^{e-3} b$ with some sign $\pm$ and some integer $b > 0$.
Now we have
\[
2^{e-3} \leq |2^{e-3} b| = |p \pm t_0| \leq p + |t_0| < 2p \enspace,
\]
therefore $e - 3 < \log_2 p + 1$ and $e < \log_2 p + 4$.
Hence the bound for $e$ holds since $e \in \mathbb{Z}$.
Then the bound for $h$ also holds since $h \in \mathbb{Z}$ as well.
\end{proof}

\begin{rem}
\label{rem:bound-Fp-square-is-tight}
When $t$ runs over the range mentioned above, the bound for $e(p^2;t)$ in Theorem \ref{thm:bound:Fp-square} is tight.
Indeed, put $f := \lfloor \log_2 p \rfloor \geq 1$ and set $t := 2(2^{f+1} - p)$.
Then $2^f \leq p < 2^{f+1}$ and hence
\[
    0 < t \leq 2 (2^{f+1} - 2^f) 
    = 2 \cdot 2^f \leq 2p \enspace,
\]
while $t \equiv 2^{f+2} \not\equiv 0 \pmod{p}$ since $p$ is odd.
Moreover, 
\[
    t^2 - 4p^2 = 4 ( (2^{f+1} - p)^2 - p^2)
    = 4 ( 2^{2f+2} - 2^{f+2} p )
    = 2^{f+4} ( 2^f - p )
\]
and $2^f - p \equiv 1 \pmod{2}$ since $f \geq 1$ and $p \equiv 1 \pmod{2}$.
Hence $\nu_2( t^2 - 4p^2 ) = f + 4 = \lfloor \log_2 p \rfloor + 4$.
\end{rem}

\subsection{Our Bound of Height in $q = p$}

We consider the case $q = p$.
We give a bound for $e(p;t) = \nu_2( t^2 - 4p )$ when $t$ runs over all integers satisfying that $|t| \leq 2 \sqrt{q} = 2 \sqrt{p}$ and $t \not\equiv 0 \pmod{p}$.
Note that when $p \geq 5$, we have $2 \sqrt{p} < p$ and hence $t$ runs over all integers with $1 \leq |t| < 2 \sqrt{p}$.

\begin{theo}
\label{thm:bound:Fp}
In the current case (with $\ell = 2$ and $q = p$), put $h := h(2;E/\mathbb{F}_p)$.
Then: 
\begin{itemize}
    \item 
    When $p \equiv 3 \pmod{4}$, we have $e(p;t) \leq 3$ and $h \leq 1$.
    \item
    When $p \equiv 5 \pmod{8}$, we have $e(p;t) \leq 4$ and $h \leq 2$.
\end{itemize}
On the other hand, when $p \equiv 1 \pmod{8}$ (hence $p \geq 17$), put $\mu := \lceil (1/2)\log_2 p \rceil \geq 3$.
We consider the following algorithm with input $p$:
\begin{enumerate}
    \item 
    $a_3 \leftarrow 1$
    \item 
    Repeat the following for $j = 4,5,\dots,\mu+1$:
    \[
        a_j \leftarrow 
        \begin{cases}
            a_{j-1} & \mbox{(if $a_{j-1}{}^2 \equiv p \pmod{2^j}$)} \\
            2^{j-2} - a_{j-1} & \mbox{(if $a_{j-1}{}^2 \not\equiv p \pmod{2^j}$)}
        \end{cases}
    \]
    \item
    \begin{itemize}
        \item 
        If $(2^{\mu} - a_{\mu+1})^2 < p$ and $\nu_2( p - a_{\mu+1}{}^2 ) < \nu_2( p - (2^{\mu} - a_{\mu+1})^2 )$, then set $b_p \leftarrow 2^{\mu} - a_{\mu+1}$
        \item
        Otherwise, set $b_p \leftarrow a_{\mu+1}$
    \end{itemize} 
    \item 
    Output $b_p$
\end{enumerate}
Then we have the following:
\begin{itemize}
    \item 
    In the algorithm, for each $j = 3,4,\dots,\mu+1$, we have $0 < a_j < 2^{j-2}$, $a_j \equiv 1 \pmod{2}$, and $p \equiv a_j{}^2 \pmod{2^j}$.
    \item 
    We have $1 \leq |2b_p| < 2 \sqrt{p}$, $e(p;2b_p) \geq \mu + 3$, and $e(p;t) \leq e(p;2b_p)$.
    Hence the maximum value $e_{\max}$ of $e(p;t)$ among all choices of $t$ is
    \[
        e_{\max} = e(p;2b_p)
        \geq \mu + 3 \enspace,
    \]
    therefore we have
    \[
        h \leq h_1 \,,\, \mbox{ where }
        h_1 := \left\lfloor \frac{ e(p;2b_p) }{ 2 } \right\rfloor
        = \left\lfloor \frac{ \nu_2( p - b_p{}^2 ) }{ 2 } \right\rfloor + 1
        \geq \left\lfloor \frac{ \mu + 3 }{ 2 } \right\rfloor
        = \left\lfloor \frac{ \lceil (1/2) \log_2 p \rceil + 3 }{ 2 } \right\rfloor \enspace.
    \]
\end{itemize}
\end{theo}
\begin{proof}
First of all, if $t$ is odd, then we have $t^2 - 4p \equiv 1 \pmod{2}$ and $e(p;t) = \nu_2(t^2 - 4p) = 0$.
Hence it suffices to consider the case $t \in 2\mathbb{Z}$ to derive a bound for $e(p;t)$.
Write $t = 2t_0$, $t_0 \in \mathbb{Z}$.
Note that $1 \leq |t_0| = |t| / 2 < \sqrt{p}$.
Now we have $t^2 - 4p = 4t_0{}^2 - 4p = 4 (t_0{}^2 - p)$ and $e(p;t) = 2 + e'(p;t_0)$ where
\[
    e'(p;t_0) := \nu_2( t_0{}^2 - p ) \enspace.
\]

When $p \equiv 3 \pmod{4}$, we have $t_0{}^2 \bmod 4 \in \{0,1\}$ and $t_0{}^2 - p \not\equiv 0 \pmod{4}$, therefore $e'(p;t_0) \leq 1$, $e(p;t) \leq 3$, and $h \leq \lfloor e(p;t)/2 \rfloor \leq 1$.
When $p \equiv 5 \pmod{8}$, we have $t_0{}^2 \bmod 8 \in \{0,1,4\}$ and $t_0{}^2 - p \not\equiv 0 \pmod{8}$, therefore $e'(p;t_0) \leq 2$, $e(p;t) \leq 4$, and $h \leq \lfloor e(p;t)/2 \rfloor \leq 2$.
Hence the claim holds for these cases.
From now on, we consider the remaining case $p \equiv 1 \pmod{8}$.
Note that the bound for $h$ in the claim will follow from the other parts of the claim.
Let $e'_{\max}$ denote the maximum value of $e'(p;t_0)$ among all choices of $t_0$.

First, we note that $0 < a_3 = 1 < 2^{3-2}$, $a_3 \equiv 1 \pmod{2}$, and $p \equiv 1 = a_3{}^2 \pmod{2^3}$, therefore the first claim holds for the case of $j = 3$.
Now suppose that $4 \leq j \leq \mu + 1$ and the first claim holds for the case of $j-1$.
If $a_{j-1}{}^2 \equiv p \pmod{2^j}$, then we have $0 < a_j = a_{j-1} < 2^{j-3} < 2^{j-2}$, $a_j = a_{j-1} \equiv 1 \pmod{2}$, and $a_j{}^2 = a_{j-1}{}^2 \equiv p \pmod{2^j}$.
On the other hand, if $a_{j-1}{}^2 \not\equiv p \pmod{2^j}$, then
\[
    p - a_{j-1}{}^2 \equiv 0 \pmod{2^{j-1}}
    \quad \mbox{and} \quad
    p - a_{j-1}{}^2 \not\equiv 0 \pmod{2^j} \enspace,
\]
therefore $p - a_{j-1}{}^2 \equiv 2^{j-1} \pmod{2^j}$.
Now we have
\[
    0 < 2^{j-2} - 2^{j-3}
    < 2^{j-2} - a_{j-1}
    = a_j
    < 2^{j-2} - 0
    = 2^{j-2} \enspace,
\]
$a_j = 2^{j-2} - a_{j-1} \equiv 0 - 1 \equiv 1 \pmod{2}$, and
\[
    a_j{}^2 = (2^{j-2} - a_{j-1})^2
    = 2^{2j-4} - 2^{j-1} a_{j-1} + a_{j-1}{}^2
    \equiv 0 - 2^{j-1} \cdot 1 + (p - 2^{j-1})
    = p - 2^j
    \equiv p \pmod{2^j} \enspace.
\]
Hence, in any case, the first claim holds for the case of $j$.
Therefore it follows recursively that the first claim holds for every $j = 3,4,\dots,\mu + 1$.

Put $a := a_{\mu + 1}$.
Then the result above shows that $a < 2^{\mu - 1}$ and
\[
    a^2 < 2^{2\mu - 2}
    \leq 2^{2 \cdot ( (1/2) \log_2 p + 1 ) - 2}
    = 2^{\log_2 p}
    = p \enspace.
\]
Therefore $|a| < \sqrt{p}$ and $a^2 - p \equiv 0 \pmod{2^{\mu+1}}$, which attains $e'(p;a) = \nu_2( a^2 - p ) \geq \mu + 1$.
Now if $e'(p;a) = e'_{\max}$, then we are in the second case $b_p \leftarrow a_{\mu + 1} = a$ for the definition of $b_p$ in the algorithm (since otherwise the choice of $t_0 := 2^{\mu} - a$ would attain $e'(p;t_0) = \nu_2( (2^{\mu} - a)^2 - p ) > \nu_2( a^2 - p ) = e'(p;a)$, a contradiction).
Hence we have $1 \leq |2b_p| = |2a| < 2 \sqrt{p}$ and $e'(p;b_p) = e'(p;a) = e'_{\max} \geq \mu + 1$, therefore
\[
    e(p;2b_p) = e'(p;b_p) + 2
    = e'_{\max} + 2
    = e_{\max}
    \geq \mu + 3 \enspace.
\]
Hence the claim holds in this case.

We consider the remaining case where $e'_{\max} > e'(p;a) \geq \mu + 1$.
Write $e'_{\max} = e'(p;c)$ with $1 \leq |c| < \sqrt{p}$.
We have $e'(p;c) = \nu_2( c^2 - p ) \geq e'(p;a) + 1 \geq \mu + 2$, therefore
\[
    c^2 \equiv p \pmod{2^{\mu+2}} \enspace,
\]
while we have
\[
    |c| < \sqrt{p}
    \leq 2^{(1/2) \log_2 p}
    \leq 2^{\mu} \enspace.
\]

To show that $a^2 \not\equiv p \pmod{2^{\mu+2}}$, assume for the contrary that $a^2 \equiv p \pmod{2^{\mu+2}}$.
Then we have
\[
    (c - a)(c + a) = c^2 - a^2
    \equiv 0 \pmod{2^{\mu+2}} \enspace,
\]
therefore
\[
    \nu_2( c - a ) + \nu_2( c + a )
    = \nu_2( (c - a)(c + a) )
    \geq \mu+2 \enspace.
\]
Now since $(c + a) - (c - a) = 2a \equiv 2 \pmod{4}$, we have either $\nu_2(c + a) \leq 1$ or $\nu_2(c - a) \leq 1$.
Therefore $\nu_2( c - a ) \geq \mu+1$ or $\nu_2( c + a ) \geq \mu+1$.
Now take the $\varepsilon \in \{\pm 1\}$ for which $c$ and $\varepsilon a$ have the same sign.
Then we have
\[
    0 < |c + \varepsilon a|
    = |c| + |a|
    < 2^{\mu} + 2^{\mu-1}
    < 2^{\mu+1} \enspace,
\]
therefore $c + \varepsilon a \not\equiv 0 \pmod{2^{\mu+1}}$ and $\nu_2(c + \varepsilon a) < \mu + 1$.
This implies that $\nu_2(c - \varepsilon a) \geq \mu + 1$, while
\[
    |c - \varepsilon a| \leq 
    \max\{ |c|, |a| \}
    < \max\{ 2^{\mu}, 2^{\mu-1} \}
    = 2^{\mu} \enspace.
\]
By combining these two properties, we have $c - \varepsilon a = 0$ and $c = \varepsilon a$.
However, now $c^2 = a^2$ and
\[
    \nu_2( a^2 - p ) = \nu_2( c^2 - p )
    = e'(p;c)
    > e'(p;a)
    = \nu_2( a^2 - p ) \enspace,
\]
a contradiction.
Hence we have $a^2 \not\equiv p \pmod{2^{\mu+2}}$.

Now we have $c^2 \equiv p \equiv a^2 \pmod{2^{\mu+1}}$ and $c^2 \equiv p \not\equiv a^2 \pmod{2^{\mu+2}}$, therefore
\[
    c^2 - a^2 \equiv 0 \pmod{2^{\mu+1}}
    \quad \mbox{and} \quad
    c^2 - a^2 \not\equiv 0 \pmod{2^{\mu+2}} \enspace,
\]
which implies that $(c - a)(c + a) = c^2 - a^2 \equiv 2^{\mu+1} \pmod{2^{\mu+2}}$ and hence
\[
    \nu_2( c - a ) + \nu_2( c + a )
    = \nu_2( (c - a)(c + a) )
    = \mu+1 \enspace.
\]
Now since $(c + a) - (c - a) = 2a \equiv 2 \pmod{4}$, we have $\nu_2(c + \varepsilon a) \leq 1$ for some $\varepsilon \in \{\pm 1\}$.
Moreover, the relation $c^2 \equiv p \pmod{2^{\mu+2}}$ implies that $c \equiv 1 \equiv a \pmod{2}$ and hence $c + \varepsilon a \equiv 0 \pmod{2}$ and $\nu_2(c + \varepsilon a) \geq 1$.
This implies that $\nu_2(c + \varepsilon a) = 1$, therefore $\nu_2(c - \varepsilon a) = (\mu + 1) - 1 = \mu$.
Hence we have $c - \varepsilon a \equiv 2^{\mu} \pmod{2^{\mu+1}}$ and $c \equiv 2^{\mu} + \varepsilon a \pmod{2^{\mu+1}}$.
Now:
\begin{itemize}
    \item 
    If $\varepsilon = 1$, then $2^{\mu} < 2^{\mu} + \varepsilon a < 2^{\mu} + 2^{\mu-1} < 2^{\mu+1}$.
    Since $|c| < 2^{\mu}$, $c$ must be $2^{\mu} + \varepsilon a - 2^{\mu+1} = a - 2^{\mu}$.
    \item
    If $\varepsilon = -1$, then $0 < 2^{\mu} + \varepsilon a < 2^{\mu}$.
    Since $|c| < 2^{\mu}$, $c$ must be $2^{\mu} + \varepsilon a = 2^{\mu} - a$.
\end{itemize}
In any case, we have $|c| = 2^{\mu} - a$, therefore $(2^{\mu} - a)^2 = c^2 < p$ and
\[
    \nu_2( p - (2^{\mu} - a)^2 ) = \nu_2( p - c^2 )
    = e'(p;c)
    > e'(p;a)
    = \nu_2( p - a^2 ) \enspace.
\]
Hence we are in the first case $b_p \leftarrow 2^{\mu} - a_{\mu+1} = 2^{\mu} - a = |c|$ for the definition of $b_p$ in the algorithm.
Now we have $1 \leq |2 b_p| = |2c| < 2 \sqrt{p}$ and $e'(p;b_p) = e'(p;|c|) = e'(p;c) = e'_{\max} \geq \mu + 2$, therefore
\[
    e(p;2b_p) = e'(p;b_p) + 2
    = e'_{\max} + 2
    = e_{\max}
    \geq \mu + 4 \enspace.
\]
Hence the claim holds in this case as well.
This completes the proof.
\end{proof}

We note that how the bound $h \leq h_1$ in Theorem \ref{thm:bound:Fp} for $p \equiv 1 \pmod{8}$ is better than a bound deduced from the fact $h = h(2;E/\mathbb{F}_p) \leq h(2;E/\mathbb{F}_{p^2})$ combined with Theorem \ref{thm:bound:Fp-square} depends on the value of $p$.
In the worst case, the value of $e_{\max} = e(p;2b_p)$ may be close to $\log_2 p$; for example, when $p = 2^{2^k} + 1$ is a Fermat prime with $k \geq 2$ (hence $p \equiv 1 \pmod{8}$), we have $e(p;2) = \nu_2( 4(p - 1) ) = 2^k + 2 = \lfloor \log_2 p \rfloor + 2$.
In such a case, the bound $h \leq h_1$ for $h = h(2;E/\mathbb{F}_p)$ given by Theorem \ref{thm:bound:Fp} (where $p \equiv 1 \pmod{8}$) has no significant advantage compared to the bound $h = h(2;E/\mathbb{F}_p) \leq h(2;E/\mathbb{F}_{p^2}) \leq h_2$ where $h_2$ is as in Theorem \ref{thm:bound:Fp-square}.
In contrast, if the value of $e_{\max} = e(p;2b_p)$ is close to the lower bound $e_{\max} \geq \mu+3$, then we have
\[
    h \leq h_1
    \approx \frac{ \mu+3 }{ 2 }
    \approx \frac{ (1/2) \log_2 p}{ 2 }
    \approx \frac{ 1 }{ 4 } \log_2 p
\]
and the bound is close to a half of the bound $h \leq h_2 \approx (1/2) \log_2 p$ given through Theorem \ref{thm:bound:Fp-square}.

\section{Computational Results}
\subsection{Average Bound of Height in $q = p$}

We investigate the heights $h = h(2;E/\mathbb{F}_p)$ of $2$-volcano graphs appearing as connected components (containing ordinary curves $E$) of the $2$-isogeny graphs $G_2(\mathbb{F}_p)$ defined over $\mathbb{F}_p$.
Table \ref{height} shows the average of the upper bounds $h_1$ (Theorem \ref{thm:bound:Fp}) of the heights $h$ for $100$ randomly generated primes $p$ with $p \equiv 1 \pmod{8}$, where $b$ denotes the bit length of $p$.
For the sake of comparison, we also include the obvious upper bounds $h \leq h_2 = \lfloor (\lfloor \log_2 p \rfloor / 2) \rfloor + 2 = \lfloor (b-1)/2 \rfloor + 2$ obtained through Theorem \ref{thm:bound:Fp-square}.

\begin{table}[htbp]
\begin{center}
\caption{Average of upper bounds $h_1$ for heights $h = h(2;E/\mathbb{F}_p)$ of $2$-volcano graphs defined over $\mathbb{F}_p$ for $100$ primes $p$ with $p \equiv 1 \pmod{8}$; here $b$ denotes the bit length of $p$, and $h_2$ denotes the obvious upper bound $\lfloor (\lfloor \log_2 p \rfloor / 2) \rfloor + 2 = \lfloor (b-1)/2 \rfloor + 2$ obtained through Theorem \ref{thm:bound:Fp-square}}\label{height}
\begin{tabular}{|c|c|c|} \hline
$b$ & Average of Bounds $h_1$ & Obvious Bound $h_2$   \\ \hline
$64$ & $18.12$ & $33$   \\ \hline
$128$ & $34.20$ & $65$  \\ \hline
$192$ & $50.21$ & $97$  \\ \hline
$256$ & $66.17$ & $129$ \\ \hline
$320$ & $82.22$ & $161$  \\ \hline
$384$ & $97.98$ & $193$  \\ \hline
$448$ & $114.25$ & $225$  \\ \hline
$512$ & $130.18$ & $257$  \\ \hline
$576$ & $146.23$ & $289$  \\ \hline
$640$ & $162.16$ & $321$  \\ \hline
$704$ & $178.24$ & $353$  \\ \hline
$768$ & $194.10$ & $385$ \\ \hline
$832$ & $210.17$ & $417$ \\ \hline
$896$ & $225.98$ & $449$  \\ \hline
$960$ & $242.13$ & $481$  \\ \hline
$1024$ & $258.05$ & $513$  \\ \hline
\end{tabular}
\end{center}
\end{table}

\subsection{Computational Time in Supersingularity Testing}
We briefly explain a deterministic supersingularity testing algorithm using an $2$-isogeny graph \cite{Sut12c}.
Firstly, we explain the following property on which the algorithm in \cite{Sut12c} is based.

\begin{prop}
If $E / \mathbb{F}_q$ is a supersingular curve, then 
$j(E) \in \mathbb{F}_{p^2}$.
\end{prop}
In contrast, when $E$ is an ordinary curve, we may have $j(E) \not\in \mathbb{F}_{p^2}$.
Accordingly, the basic strategy of the algorithm in \cite{Sut12c} is to search (by utilizing the structure of $2$-volcano graphs) for a vertex $E'$ in the $2$-isogeny graph with $j(E) \not\in \mathbb{F}_{p^2}$; $E$ is ordinary if such a curve $E'$ is found, while $E$ is supersingular if such a curve $E'$ is not found.

The algorithm in \cite{Sut12c} determines supersingularity of an elliptic curve $E$ as follows.
\begin{enumerate}
\item We compute $3$ outgoing edges from the $j$-invariant of a given elliptic curve $E / \mathbb{F}_q$ on the $2$-isogeny graph $G_2(E / \mathbb{F}_q)$ 
to get next $3$ vertices $E_1,E_2,E_3$.
\item We iteratively compute $3$ paths $P_1,P_2,P_3$ 
without backtracking in parallel, where the first edge of $P_i$ is $E \to E_i$. Then, we determine supersingularity of $E$ by computing edges ($2$-isogenies) 
in $\lfloor \log_2 p \rfloor + 1$ steps as follows.
\begin{enumerate}
\item We determine the given elliptic curve $E$ as ordinary one if a vertex $E'$ satisfying $j(E') \not\in \mathbb{F}_{p^2}$ appears during the computation.
\item Otherwise, we 
determine the given elliptic curve $E$ as supersingular one. 
\end{enumerate}
\end{enumerate}
The original algorithm uses 
a classical modular polynomial to compute $2$-isogenies.
On the other hand, the supersingularity testing algorithm in \cite{HN21} reduces 
the computational cost by about half compared to the algorithm in \cite{Sut12c}, by using some property of Legendre curves instead of modular polynomials. 
Now note that the number $\lfloor \log_2 p \rfloor + 1$ of steps in the algorithm originates from the known upper bound $h \leq \log_2(\sqrt{4q}) = \log_2 p + 1$ of the height $h$ of the $2$-isogeny graph for the ordinary case.
Therefore, the number of required steps is reduced once we replace the upper bound of $h$ with the bound $h \leq h_2$ given by Theorem \ref{thm:bound:Fp-square}.

We investigate the performance of the algorithm in \cite{HN21} for new bound of their heights.
We denote by $b$ the bit-length of $p$. 
For $100$ prime numbers of $b$-bit length, we randomly selected  a supersingular curve for each prime (here we used supersingular curves only, because they correspond to the worst case in computation time of the algorithm).
We computed the algorithm's performance separately for primes where 
$p \equiv 1 \pmod{4}$ and $p \equiv 3 \pmod{4}$ by using Magma.
The environment for the experiment is: Ubuntu 20.04.5 LTS, Intel Core i7-11700K @ 3.60GHz (16 cores), 128GB memory, Magma v2.26-11.
Table \ref{time1} shows that the computational time of the supersingularity testing algorithm is reduced by about half, which is consistent to the improvement of the upper bound for the heights by about half. 

\begin{table}[htbp]
\begin{center}
\caption{Average execution times of a superingularity testing algorithm in \cite{HN21} for the known bound $h_0$ and our improved bound $h_2$ for the heights, 
where $h_0 = 
\left\lfloor \log_2 p \right\rfloor + 1, h_2 = \left\lfloor \frac{ 1 }{ 2 } \lfloor \log_2 p \rfloor \right\rfloor + 2$ and $b$ denotes the bit length of $p = 4r + 1$ or $4r + 3, r \in \mathbb{Z}$ (CPU times in milliseconds)}\label{time1}
\begin{tabular}{|c|c|c|c|c|} \hline
$b$ & $h_0 
$ ($p = 4r + 1$) & $h_2$ ($p = 4r + 1$) & $h_0 
$ ($p = 4r + 3$) & $h_2$ ($p = 4r + 3$)  \\ \hline
$64$ & $18$ & $10$ & $12$ & $8$   \\ \hline
$128$ & $66$ & $37$ & $51$ & $27$   \\ \hline
$192$ & $158$ & $84$ & $126$ & $67$   \\ \hline
$256$ & $315$  & $165$ & $249$  & $132$ \\ \hline
$320$ & $537$ & $278$ & $430$ & $225$   \\ \hline
$384$ & $880$ & $453$ & $694$ & $359$   \\ \hline
$448$ & $1361$ & $698$ & $1072$ & $554$   \\ \hline
$512$ & $2016$ & $1032$ & $1576$ & $807$  \\ \hline
$576$ & $2897$ & $1475$ & $2230$ & $1142$  \\ \hline
$640$ & $3822$ & $1947$ & $3006$ & $1538$   \\ \hline
$704$ & $5157$ & $2627$ & $4012$ & $2044$   \\ \hline
$768$ & $6716$ & $3414$ & $5275$ & $2682$  \\ \hline
$832$ & $8738$ & $4427$ & $6734$ & $3428$  \\ \hline
$896$ & $11229$ & $5694$ & $8594$ & $4365$   \\ \hline
$960$ & $13879$ & $7016$ & $10720$ & $5438$   \\ \hline
$1024$ & $17144$ & $8679$ & $13121$ & $6639$   \\ \hline
\end{tabular}
\end{center}
\end{table}

\paragraph*{Acknowledgements.}
This work was supported by JSPS KAKENHI Grant Numbers JP22K11906 and JP24K17281, Japan.
This work was supported by Institute of Mathematics for Industry, Joint Usage/Research Center in Kyushu University. (FY2024 Short-term Visiting Researcher \lq\lq Towards improving security reductions in isogeny-based cryptosystems\rq\rq{} (2024a026).)

\bibliographystyle{abbrv}
\bibliography{vol}

\end{document}